\documentclass[10pt,reqno]{amsart}
\usepackage[hypertex]{hyperref}

\usepackage{amsmath,amsthm}
\usepackage{amssymb}
\usepackage{euscript}

\numberwithin{equation}{subsection}

\newcommand{\sqsp}{\renewcommand{\baselinestretch}{1.15}\tiny\normalsize}

\newcommand{\relphantom[1]}{\mathrel{\phantom{#1}}}

\newtheorem{theorem}[subsection]{Theorem}
\newtheorem{lemma}[subsection]{Lemma}
\newtheorem{proposition}[subsection]{Proposition}
\newtheorem{corollary}[subsection]{Corollary}

\theoremstyle{definition}
\newtheorem{definition}[subsection]{Definition}

\newcommand{\bk}{\mathbf{k}}

\newcommand{\mualpha}{\mu_\alpha}
\newcommand{\mubeta}{\mu_\beta}

\newcommand{\andspace}{\quad\text{and}\quad}
\newcommand{\defn}{\buildrel \text{def} \over =}

\begin{document}

\title{The Mikheev identity in right Hom-alternative algebras}
\author{Donald Yau}

\begin{abstract}
It is shown that in every multiplicative right Hom-alternative algebra, a Hom-type generalization of the Mikheev identity holds.  It is then inferred that a multiplicative right Hom-alternative algebra with an injective twisting map and without Hom-nilpotent elements or left zero-divisors must be a Hom-alternative algebra.
\end{abstract}



\address{Department of Mathematics\\
    The Ohio State University at Newark\\
    1179 University Drive\\
    Newark, OH 43055, USA}
\email{dyau@math.ohio-state.edu}

\date{\today}
\maketitle

\sqsp

\section{Introduction}

Right and left alternative algebras are defined by the identities
\begin{equation}
\label{rightalternative}
(ab)b = a(bb) \andspace (aa)b = a(ab),
\end{equation}
respectively.  An alternative algebra is an algebra that is both left alternative and right alternative.  Alternative algebras are important for purely algebraic reasons as well as for applications in other fields.  For example, the $8$-dimensional Cayley algebras are alternative algebras that are not associative  \cite{schafer}.  Right alternative algebras give rise to Jordan algebras via the anti-commutator product \cite{albert3}, and they are power associative \cite{albert1,albert2}.  Likewise, alternative algebras give rise to Maltsev algebras via the commutator bracket \cite{maltsev}.  Moreover, alternative algebras play an important role in project geometry \cite{moufang,skor2} and generalized polygons \cite{tw}.

Twisted generalizations of (left/right) alternative algebras, called (left/right) Hom-alternative algebras, were introduced by Makhlouf \cite{mak} by relaxing the defining identities in (left/right) alternative algebras by a linear self-map $\alpha$, called the twisting map.  Construction results and examples of Hom-alternative algebras can be found in \cite{mak,yau4}.  Hom-Lie algebras first appeared in \cite{hls} to describe deformations of the Witt algebra and the Virasoro algebra.  Hom-associative algebras were defined in \cite{ms} and further studied in \cite{ms2,yau0,yau1}.  For other Hom-type algebras, see \cite{yau2,yau3,yau4,yau5,yau6} and the references therein.

Many properties of (right) alternative algebras have been generalized to (right) Hom-alternative algebras by the author.  A Hom-algebra $(A,\mu,\alpha)$ is said to be multiplicative if the twisting map $\alpha$ is multiplicative with respect to the multiplication $\mu$.  The author proved in \cite{yau4} that multiplicative Hom-alternative algebras are Hom-Maltsev admissible, and they satisfy Hom-versions of the Moufang identities.  In \cite{yau6} the author showed that multiplicative right Hom-alternative algebras are Hom-power associative and Hom-Jordan admissible.  Actually, Hom-Jordan admissibility was also established in \cite{yau4} for multiplicative Hom-alternative algebras.  The concept of Hom-power associativity was introduced by the author in \cite{yau5} as a twisted generalization of power associativity \cite{albert1,albert2}.

The purpose of this paper is to establish criteria that guarantee that a right Hom-alternative algebra is Hom-alternative.  Examples of right Hom-alternative algebras that are not left Hom-alternative can be found in \cite{yau6}.   A description of our main results follows.

The basis of our results is a theorem by Mikheev \cite{mikheev}, which we now recall.  The right and left alternative identities \eqref{rightalternative} can be written as
\begin{equation}
\label{alternative}
(a,b,b) = 0 \andspace (a,a,b) = 0,
\end{equation}
respectively, where
\[
(x,y,z) = (xy)z - x(yz)
\]
is the associator.  In \cite{mikheev} Mikheev proved that in every right alternative algebra, the identity
\begin{equation}
\label{mikheevid}
(a,a,b)^4 = 0
\end{equation}
holds.  Roughly speaking, the \emph{Mikheev identity} \eqref{mikheevid} says that every right alternative algebra is fairly close to being left alternative.  Here the fourth power is defined as
\[
x^4 = ((xx)x)x.
\]
In fact, the order of the parentheses does not matter because every right alternative algebra is power associative.

Our main result is the Hom-version of the Mikheev identity \eqref{mikheevid}.  In a Hom-algebra $(A,\mu,\alpha)$, the role of the associator is played by the \emph{Hom-associator}
\[
(x,y,z) = (xy)\alpha(z) - \alpha(x)(yz).
\]
Right and left Hom-alternative algebras are defined as in \eqref{alternative} using the Hom-associator.  A Hom-alternative algebra is both left and right Hom-alternative.  The \emph{Hom-power} $x^n$ \cite{yau5} in a Hom-algebra is defined inductively as
\[
x^n = x^{n-1}\alpha^{n-2}(x).
\]
We can now state our main result, which is a Hom-type generalization of the Mikheev identity \eqref{mikheevid} in right alternative algebras.

\begin{theorem}
\label{thm:mikheev}
Let $(A,\mu,\alpha)$ be a multiplicative right Hom-alternative algebra.  Then the identity
\begin{equation}
\label{mikheev}
\alpha^6\left((a,a,b)^4\right) = 0
\end{equation}
holds in $A$, where $(,,)$ is the Hom-associator.
\end{theorem}

Theorem \ref{thm:mikheev} says that in a multiplicative right Hom-alternative algebra, although the Hom-associator $(a,a,b)$ does not have to be zero, it is nonetheless the case that its fourth Hom-power $(a,a,b)^4$ lies in the kernel of $\alpha^6$.

We now discuss some consequences of Theorem \ref{thm:mikheev}.  The following result is the special case when the twisting map is injective.

\begin{corollary}
\label{cor0:mikheev}
Let $(A,\mu,\alpha)$ be a multiplicative right Hom-alternative algebra in which $\alpha$ is injective.  Then the identity
\[
(a,a,b)^4 = 0
\]
holds in $A$, where $(,,)$ is the Hom-associator.
\end{corollary}

Restricting to the case $\alpha = Id_A$ in Corollary \ref{cor0:mikheev}, one recovers the Mikheev identity \eqref{mikheevid} in a right alternative algebra \cite{mikheev}.

Recall that a division algebra is a unital algebra in which $1 \not= 0$ and every non-zero element has a two-sided multiplicative inverse.  In a right alternative division algebra, it is true that $x^4 = 0$ implies $x = 0$.  Indeed, if $y^2 = 0$, then $y = 0$ because, otherwise, $y^{-1}$ exists.  But right alternativity then implies
\[
0 = y^{-1}(yy) = (y^{-1}y)y = y,
\]
which is a contradiction.  Now by right alternativity again, we have
\[
0 = x^4 = (x^2,x,x) + x^2x^2 = x^2x^2,
\]
so $x^2 = 0$, and hence $x = 0$.  In particular, from the $\alpha = Id_A$ special case of Corollary \ref{cor0:mikheev}, one recovers the following result of Skornyakov \cite{skor1}.

\begin{corollary}
Every right alternative division algebra is an alternative algebra.
\end{corollary}

In a Hom-algebra, define a \emph{Hom-nilpotent element} to be a non-zero element $x$ such that the Hom-power $x^n$ is $0$ for some $n \geq 2$.  In the context of Corollary \ref{cor0:mikheev}, if the Hom-associator $(a,a,b)$ is non-zero for some elements $a$ and $b$, then it is a Hom-nilpotent element.  We now have the following special case of Corollary \ref{cor0:mikheev}, which gives conditions under which a right Hom-alternative algebra is Hom-alternative.

\begin{corollary}
\label{cor1:mikheev}
Let $(A,\mu,\alpha)$ be a multiplicative right Hom-alternative algebra.  Suppose that $\alpha$ is injective and that $A$ has no Hom-nilpotent elements.  Then $A$ is a Hom-alternative algebra.
\end{corollary}

Setting $\alpha = Id_A$ in Corollary \ref{cor1:mikheev}, one recovers the following result of Kleinfeld \cite{kleinfeld}.

\begin{corollary}
\label{cor:kleinfeld}
Every right alternative algebra without nilpotent elements is an alternative algebra.
\end{corollary}

We now discuss a special case of Corollary \ref{cor1:mikheev}.  Zero-divisors in Hom-algebras are defined as usual.  More precisely, a non-zero element $a$ in a Hom-algebra $(A,\mu,\alpha)$ is called a \emph{left (resp., right) zero-divisor} if there exists a non-zero element $b$ such that
\[
ab = 0 \quad\text{(resp., $ba = 0$)}.
\]
A Hom-algebra has a left zero-divisor if and only if it has a right zero-divisor.  If the twisting map $\alpha$ is injective, then every Hom-nilpotent element gives rise to a left zero-divisor.  Indeed, suppose $a$ is a Hom-nilpotent element and $n \geq 2$ is the least integer such that
\[
a^n = a^{n-1}\alpha^{n-2}(a) = 0.
\]
Then $a^{n-1} \not= 0$ by the choice of $n$, and $\alpha^{n-2}(a) \not= 0$ by the injectivity of $\alpha$.  Thus, the Hom-power $a^{n-1}$ is a left zero-divisor.  Therefore, when $\alpha$ is injective, the absence of left zero-divisors implies the absence of Hom-nilpotent elements.   We now have the following special case of Corollary \ref{cor1:mikheev}.

\begin{corollary}
\label{cor2:mikheev}
Let $(A,\mu,\alpha)$ be a multiplicative right Hom-alternative algebra.  Suppose that $\alpha$ is injective and that $A$ has no left zero-divisors.  Then $A$ is a Hom-alternative algebra.
\end{corollary}

Setting $\alpha = Id_A$ in Corollary \ref{cor2:mikheev}, one recovers the following special case of Corollary \ref{cor:kleinfeld}.

\begin{corollary}
Every right alternative algebra without left zero-divisors is an alternative algebra.
\end{corollary}

This finishes the discussion of our results.  The rest of this paper is organized as follows.  In the next section, we recall some basic definitions regarding Hom-algebras and some properties of right Hom-alternative algebras.  In section \ref{sec:ex} we illustrate Corollary \ref{cor0:mikheev} by exhibiting an infinite family of mutually non-isomorphic multiplicative right Hom-alternative algebras in which the twisting maps are injective and $(a,a,b)^2 \not= 0$ for some elements $a$ and $b$.  The proof of Theorem \ref{thm:mikheev} is given in section \ref{sec:main}.

\section{Basic properties of right Hom-alternative algebras}
\label{sec:basic}

The purpose of this section is to recall some properties of right Hom-alternative algebras that are needed for the proof of Theorem \ref{thm:mikheev}.

\subsection{Notations}

We work over a fixed field $\bk$ of characteristic $0$.  Given elements $x$ and $y$ in a $\bk$-module $A$ and a bilinear map $\mu \colon A^{\otimes 2} \to A$, we often abbreviate $\mu(x,y)$ to $xy$.  For a linear self-map $\alpha \colon A \to A$, denote by $\alpha^n$ the $n$-fold composition of $n$ copies of $\alpha$, with $\alpha^0 \equiv Id$.  The alphabets $a,b,c,d,w,x,y$, and $z$ will denote arbitrary elements in the $\bk$-module under consideration.

We first recall some basic definitions regarding Hom-algebras.

\begin{definition}
\label{def:homalg}
A \textbf{Hom-algebra} is a triple $(A,\mu,\alpha)$ in which $A$ is a $\bk$-module, $\mu \colon A^{\otimes 2} \to A$ is a bilinear map, called the multiplication, and $\alpha \colon A \to A$ is a linear map, called the twisting map.  A Hom-algebra $(A,\mu,\alpha)$ is often abbreviated to $A$.  A Hom-algebra $(A,\mu,\alpha)$ is \textbf{multiplicative} if $\alpha$ is multiplicative with respect to $\mu$, i.e., $\alpha\mu = \mu \alpha^{\otimes 2}$.
\end{definition}

From now on, an algebra $(A,\mu)$ is also regarded as a Hom-algebra $(A,\mu,Id)$ with identity twisting map, unless otherwise specified.  Next we recall the Hom-type generalizations of the associator and (left/right) alternative algebras.

\begin{definition}
\label{def:homassociator}
Let $(A,\mu,\alpha)$ be a Hom-algebra.
\begin{enumerate}
\item
Define the \textbf{Hom-associator} \cite{ms} $(,,)_A \colon A^{\otimes 3} \to A$ by
\begin{equation}
\label{homassociator}
(x,y,z)_A = (xy)\alpha(z) - \alpha(x)(yz).
\end{equation}
\item
A \textbf{right Hom-alternative algebra} is defined by the identity
\begin{equation}
\label{rhomalt}
(x,y,y)_A = 0.
\end{equation}
A \textbf{left Hom-alternative algebra} is defined by the identity
\[
(x,x,y)_A = 0.
\]
A \textbf{Hom-alternative algebra} \cite{mak} is a Hom-algebra that is both left Hom-alternative and right Hom-alternative.
\end{enumerate}
\end{definition}

When there is no danger of confusion, we will omit the subscript $A$ in the Hom-associator $(,,)_A$.  When the twisting map $\alpha$ is the identity map, the above definitions reduce to the usual definitions of the associator and (left/right) alternative algebras.  In expanded form, right Hom-alternativity \eqref{rhomalt} means the identity
\begin{equation}
\label{xyy}
(xy)\alpha(y) = \alpha(x)(yy)
\end{equation}
holds.

Next we recall from \cite{yau5} the Hom-type generalization of right powers \cite{albert1,albert2}.

\begin{definition}
\label{def:hompower}
Let $(A,\mu,\alpha)$ be a Hom-algebra and $x$ be an element in $A$.  Define the \textbf{$n$th Hom-power} $x^n \in A$ inductively by
\begin{equation}
\label{hompower}
x^1 = x \andspace
x^n = x^{n-1}\alpha^{n-2}(x)
\end{equation}
for $n \geq 2$.
\end{definition}

We have
\begin{equation}
\label{x4}
x^2 = xx,\quad
x^3 = (xx)\alpha(x), \andspace
x^4 = ((xx)\alpha(x))\alpha^2(x).
\end{equation}
Multiplicative right Hom-alternative algebras are Hom-power associative \cite{yau5,yau6}.  However, we do not need that result in this paper.

In the rest of this section, we recall some properties of right Hom-alternative algebras that will be used in later sections.  The following result gives the linearization of the right Hom-alternative identity.

\begin{lemma}[\cite{mak}]
\label{lem:linearized}
Let $(A,\mu,\alpha)$ be a Hom-algebra.  Then the following statements are equivalent.
\begin{enumerate}
\item
$A$ is right Hom-alternative, i.e., $(x,y,y) = 0$.
\item
$A$ satisfies
\begin{equation}
\label{righthomalt}
(x,y,z) = -(x,z,y).
\end{equation}
\end{enumerate}
\end{lemma}

In a right alternative algebra, it is known that the \emph{right Moufang identity}
\[
((xy)z)y = x((yz)y)
\]
holds \cite{kleinfeld}.  The following result is the Hom-version of this identity.

\begin{lemma}
\label{lem:moufang}
Let $(A,\mu,\alpha)$ be a multiplicative right Hom-alternative algebra.  Then the \textbf{right Hom-Moufang identity}
\begin{equation}
\label{rightmoufang}
\left[(xy)\alpha(z)\right]\alpha^2(y) = \alpha^2(x)\left[(yz)\alpha(y)\right]
\end{equation}
holds.
\end{lemma}

\begin{proof}
This result was proved in \cite{yau6}.  Here is a sketch of the proof.  In any multiplicative Hom-algebra, the \emph{Hom-Teichm\"{u}ller identity}
\[
\begin{split}
f(w,x,y,z) &\defn (wx,\alpha(y),\alpha(z)) - (\alpha(w),xy,\alpha(z)) + (\alpha(w),\alpha(x),yz)\\
&\relphantom{} - \alpha^2(w)(x,y,z) - (w,x,y)\alpha^2(z)\\
&= 0
\end{split}
\]
holds, which one can show by expanding the five Hom-associators.  Using the Hom-Teichm\"{u}ller identity, one shows that
\begin{equation}
\label{xyyz}
(\alpha(x),\alpha(y),yz) = (x,y,z)\alpha^2(y)
\end{equation}
by expanding the right-hand side of
\[
0 = f(x,y,y,z) - f(x,z,y,y) + f(x,y,z,y).
\]
Finally, one proves the right Hom-Moufang identity by computing $\left[(xy)\alpha(z)\right]\alpha^2(y)$ using \eqref{homassociator}, \eqref{xyyz}, \eqref{righthomalt}, and then \eqref{homassociator} again.
\end{proof}

The category of right Hom-alternative algebras is closed under twisting by self-weak morphisms.  Let us first recall the relevant definitions.

\begin{definition}
\label{def:morphism}
Let $A$ and $B$ be Hom-algebras.  A \textbf{weak morphism} $f \colon A \to B$ of Hom-algebras is a linear map $f \colon A \to B$ such that $f\mu_A = \mu_B f^{\otimes 2}$. A \textbf{morphism} $f \colon A \to B$ is a weak morphism such that $f\alpha_A = \alpha_Bf$.  The Hom-algebras $A$ and $B$ are \textbf{isomorphic} if there exists an invertible morphism $f \colon A \to B$.
\end{definition}

The following result was proved in \cite{yau6}.  It says that the category of right Hom-alternative algebras is closed under twisting by weak morphisms.

\begin{theorem}
\label{thm:twist}
Let $(A,\mu,\alpha)$ be a right Hom-alternative algebra, and let $\beta \colon A \to A$ be a weak morphism.  Then the Hom-algebra
\begin{equation}
\label{abeta}
A_\beta = (A,\mubeta = \beta \mu, \beta\alpha)
\end{equation}
is also a right Hom-alternative algebra.  Moreover, if $A$ is multiplicative and $\beta$ is a morphism, then $A_\beta$ is multiplicative.
\end{theorem}

Indeed, given any Hom-algebra $A$ and weak morphism $\beta \colon A \to A$, their Hom-associators are related as follows:
\begin{equation}
\label{beta2}
\beta^2 (x,y,z)_A = (x,y,z)_{A_\beta}.
\end{equation}
Therefore, if $A$ is right Hom-alternative, then so is $A_\beta$.

Setting $\alpha = Id_A$ in Theorem \ref{thm:twist}, one recovers \cite{mak} (Theorem 3.1), which says that multiplicative right Hom-alternative algebras may arise from right alternative algebras and their morphisms.  A twisting result of this form was first given by the author in \cite{yau1} for $G$-Hom-associative algebras, which include Hom-associative and Hom-Lie algebras.

\begin{corollary}
\label{cor2:twist}
Let $(A,\mu)$ be a right alternative algebra and $\beta \colon A \to A$ be an algebra morphism.  Then
\[
A_\beta = (A,\mubeta = \beta\mu,\beta)
\]
is a multiplicative right Hom-alternative algebra.
\end{corollary}

In the next section, we will use Corollary \ref{cor2:twist} to construct multiplicative right Hom-alternative algebras from a given right alternative algebra.

\section{Examples}
\label{sec:ex}

The purpose of this section is to construct examples of right Hom-alternative algebras that are not right alternative and in which $(a,a,b)^2 \not= 0$ for some elements $a$ and $b$.

We start with a right alternative algebra constructed by Mikheev \cite{mikheev}.  Consider the $13$-dimensional algebra $A$ with a basis $\{e_i \colon 1 \leq i \leq 13\}$.  The non-zero binary products among the basis elements are:
\[
\begin{split}
e_1e_1 &= e_3,\quad e_1e_2 = e_4,\quad e_1e_3 = e_5, \quad e_1e_4 = e_8, \quad e_1e_6 = e_9,\\
e_1e_7 &= e_{12} = e_1e_9, \quad e_1e_{10} = e_{11},\\
e_2e_1 &= e_6, \quad e_2e_3 = e_{10}, \quad e_3e_1 = e_5, \quad e_3e_2 = e_7, \quad e_3e_6 = e_{11} + e_{12},\\
e_4e_1 &= -e_7 + e_8 + e_9, \quad e_5e_2 = e_{11} + e_{12}, \quad e_6e_1 = e_{10},\\
e_8e_7 &= e_{13} = e_8e_9 = -e_8e_{10},\quad
e_9e_7 = -e_{13} = e_9e_9 = -e_9e_{10},\\
e_{11}e_4 &= e_{13} = -e_{11}e_6,\quad
e_{12}e_4 = -e_{13} = -e_{12}e_6.
\end{split}
\]
Mikheev showed in \cite{mikheev} that $A$ is a right alternative algebra such that
\[
(e_1,e_1,e_2)_A = e_7 - e_8 \andspace
(e_1,e_1,e_2)_A^2 = -e_{13}.
\]
We will apply Corollary \ref{cor2:twist} to the right alternative algebra $A$ with suitable morphisms.

Let us now construct a family of morphisms on $A$.  Let $\lambda$ and $\xi$ be scalars in $\bk$ with $\lambda\xi \not= 0$ and $\lambda \not= \xi$.  Consider the linear map $\alpha = \alpha_{\lambda,\xi} \colon A \to A$ defined by:
\[
\begin{split}
\alpha(e_1) &= \lambda e_1,\quad \alpha(e_2) = \xi e_2, \quad \alpha(e_3) = \lambda^2 e_3, \\
\alpha(e_4) &= \lambda\xi e_4, \quad
\alpha(e_5) = \lambda^3 e_5, \quad \alpha(e_6) = \lambda\xi e_6,\\
\alpha(e_i) &= \lambda^2\xi e_i \quad\text{for $i=7,8,9,10$},\\
\alpha(e_j) &= \lambda^3\xi e_j \quad\text{for $j=11,12$}, \quad \alpha(e_{13}) = \lambda^4\xi^2 e_{13}.
\end{split}
\]
It is straightforward to check that $\alpha_{\lambda,\xi}$ preserves all the defining relations among the $e_i$'s, so $\alpha_{\lambda,\xi}$ is an algebra morphism on $A$.  By Corollary \ref{cor2:twist} there is a multiplicative right Hom-alternative algebra
\[
A(\lambda,\xi) = A_{\alpha_{\lambda,\xi}} = (A, \alpha_{\lambda,\xi}\mu, \alpha_{\lambda,\xi}),
\]
where $\mu$ is the original multiplication on $A$. Note that $\alpha_{\lambda,\xi}$ is an isomorphism, so $A(\lambda,\xi)$ satisfies the hypotheses of Corollary \ref{cor0:mikheev}.

We aim to prove the following three statements:
\begin{enumerate}
\item
$(A, \alpha_{\lambda,\xi}\mu)$ is not a right alternative algebra.
\item
$(e_1,e_1,e_2)^2_{A(\lambda,\xi)} \not= 0$.
\item
The Hom-algebras $A(\lambda,\xi)$ and $A(\lambda',\xi')$ are not isomorphic (Definition \ref{def:morphism}), provided that one of the following conditions holds:
\[
\lambda \not= \lambda'^r\xi'^s, \quad \xi \not= \lambda'^r\xi'^s, \quad \lambda' \not= \lambda^r\xi^s, \quad\text{or}\quad \xi' \not= \lambda^r\xi^s
\]
for any $r \in \{0,\ldots,4\}$ and $s \in \{0,1,2\}$.
\end{enumerate}

To prove that $(A, \alpha_{\lambda,\xi}\mu)$ is not a right alternative algebra, let $\mualpha$ denote the multiplication $\alpha_{\lambda,\xi}\mu$.  Then we have
\[
\mualpha(\mualpha(e_2,e_1),e_1) = \lambda^3\xi^2 e_{10}
\andspace
\mualpha(e_2,\mualpha(e_1,e_1)) = \lambda^4\xi e_{10}.
\]
The two elements above are equal if and only
\[
\lambda^3\xi^2 = \lambda^4\xi,
\]
or equivalently $\lambda = \xi$.  This last condition does not happen by our choices of $\lambda$ and $\xi$.  We conclude that $(A, \alpha_{\lambda,\xi}\mu)$ is not a right alternative algebra.

To prove the second assertion, we use \eqref{beta2} and compute as follows:
\[
\begin{split}
(e_1,e_1,e_2)^2_{A(\lambda,\xi)}
&= \left\{\alpha^2(e_1,e_1,e_2)_A\right\}^2\\
&= \left\{\alpha^2(e_7-e_8)\right\}^2\\
&= \left\{\lambda^4\xi^2(e_7 - e_8)\right\}^2\\
&= -\lambda^8\xi^4 e_{13}.
\end{split}
\]
Since the above element is non-zero, we have proved the second assertion.

To prove the last assertion, suppose that $g \colon A(\lambda,\xi) \to A(\lambda',\xi')$ is a Hom-algebra isomorphism.  Write
\[
g(e_i) = \sum_j \eta_{ij} e_j, \quad\alpha = \alpha_{\lambda,\xi}, \quad\text{and}\quad \alpha' = \alpha_{\lambda',\xi'}.
\]
Comparing $g\alpha(e_1)$ with $\alpha'g(e_1)$, we have that
\[
\lambda\left(\sum\eta_{1j}e_j\right) = \sum \eta_{1j}\alpha'(e_j) = \sum \eta_{1j}\lambda'^r\xi'^s e_j.
\]
Since at least one $\eta_{1j}$ is non-zero, we infer that
\[
\lambda = \lambda'^r\xi'^s
\]
for some $r \in \{0,\ldots,4\}$ and $s \in \{0,1,2\}$.  If we use $e_2$ instead of $e_1$, then an analogous computation shows that
\[
\xi = \lambda'^r\xi'^s
\]
for some $r \in \{0,\ldots,4\}$ and $s \in \{0,1,2\}$.  By symmetry we infer that both $\lambda'$ and $\xi'$ have the form $\lambda^r\xi^s$.  This proves the third assertion.

\section{Proof of the main result}
\label{sec:main}

The purpose of this section is to prove Theorem \ref{thm:mikheev}.  Our proof is modeled on the proof of the Mikheev identity given in \cite{zsss} (Chapter 16).

To simplify the typography, let us first extend the Herstein notations \cite{herstein} to Hom-algebras.  As in most of the classic literature, we allow linear maps, including the twisting map of a Hom-algebra, to act from the right.

\begin{definition}
\label{def:herstein}
Let $(A,\mu,\alpha)$ be a Hom-algebra, $a,b$ be elements in $A$, and $n$ be a non-negative integer.  Define:
\begin{enumerate}
\item
$a_n$ to be $\alpha^n(a)$;
\item
$a'$ to be the operator of right multiplication by $a$;
\item
$a^b$ and $a_b$ to be the operators
\begin{equation}
\label{ab}
a^b = \alpha (ab)' - a'b_1' \andspace
a_b = \alpha (ab)' - b'a_1',
\end{equation}
respectively.
\end{enumerate}
\end{definition}

Since these operators act from the right, in terms of an element $x \in A$, we have
\begin{equation}
\label{herstein}
\begin{split}
xa' &= xa,\\
xa^b &= \alpha(x)(ab) - (xa)\alpha(b) = -(x,a,b),\\
xa_b &= \alpha(x)(ab) - (xb)\alpha(a) = -(x,a,b) + (xa)\alpha(b) - (xb)\alpha(a),
\end{split}
\end{equation}
where $(,,)$ is the Hom-associator in $A$ \eqref{homassociator}.

\begin{proof}[Proof of Theorem \ref{thm:mikheev}]
Pick arbitrary elements $a$ and $b$ in $A$, and write $p = (a,a,b)$.  Then we have
\begin{equation}
\label{4}
p = -(-a,a,b) = -aa^b
\end{equation}
by \eqref{herstein}.  To prove the desired identity \eqref{mikheev}, we compute as follows:
\[
\begin{split}
\alpha^6\left[(a,a,b)^4\right] &= \alpha^6\left\{[(pp)\alpha(p)]\alpha^2(p)\right\} \quad\text{(by \eqref{x4})}\\
&= pp'p_1'p_2'\alpha^6 \quad\text{(Definition \ref{def:herstein})}\\
&= -aa^bp'p_1'p_2'\alpha^6 \quad\text{(by \eqref{4})}.
\end{split}
\]
In Proposition \ref{prop:0} below, we will show that the operator in the previous line, $a^bp'p_1'p_2'\alpha^6$, is $0$.  Therefore, the proof of Theorem \ref{thm:mikheev} will be complete once Proposition \ref{prop:0} is proved.
\end{proof}

The rest of this section is devoted to the proof of Proposition \ref{prop:0} below, which involves a series of preliminary lemmas regarding the operators $a'$, $a^b$, and $a_b$.  Briefly, the proof of Proposition \ref{prop:0} is divided into two parts:
\begin{enumerate}
\item
The first part consists of showing that the operator $a^bp'p_1'p_2'\alpha^6$ is the sum of two operators \eqref{d+e}.
\item
The second part consists of showing that each summand in \eqref{d+e} is trivial (Lemmas \ref{lem:d} and \ref{lem:e}).
\end{enumerate}

Let us begin with the linearity and multiplicative properties of the operators defined above.

\begin{lemma}
\label{lem1}
Let $(A,\mu,\alpha)$ be a Hom-algebra, $a,b$, and $c$ be elements in $A$, and $n$ and $m$ be non-negative integers.  Then the following statements hold.
\begin{enumerate}
\item
$(a+b)_n = a_n + b_n$ and $(a_n)_m = a_{n+m}$.
\item
The operators $a'$, $a^b$, and $a_b$ are linear in both $a$ and $b$.
\item
If $A$ is multiplicative, then
\begin{equation}
\label{0}
\begin{split}
(ab)_n &= a_nb_n, \quad (a,b,c)_n = (a_n,b_n,c_n)\\
a'\alpha^n &= \alpha^n a_n', \quad
a^b \alpha^n = \alpha^n a_n^{b_n},\andspace
a_b \alpha^n = \alpha^n {a_n}_{b_n}.
\end{split}
\end{equation}
\end{enumerate}
\end{lemma}

\begin{proof}
The first assertion simply says that $\alpha^n$ is linear and $\alpha^n\alpha^m = \alpha^{n+m}$.  The linearity of the three operators follows from the definition \eqref{herstein}.  Note that if $A$ is multiplicative, then $\alpha^n$ is multiplicative with respect to $\mu$, which implies the first line in \eqref{0}.  For the second line in \eqref{0}, we use \eqref{herstein} and compute as follows:
\[
\begin{split}
xa'\alpha^n &= \alpha^n(xa) = \alpha^n(x)\alpha^n(a) = x\alpha^n a_n',\\
xa^b\alpha^n &= -\alpha^n(x,a,b) = -(\alpha^n(x),\alpha^n(a),\alpha^n(b)) = x\alpha^n a_n^{b_n}.
\end{split}
\]
A similar calculation proves the last identity in \eqref{0}.
\end{proof}

The following result gives the operator forms of the right alternative identity \eqref{rhomalt} and the right Hom-Moufang identity \eqref{rightmoufang}.

\begin{lemma}
\label{lem2}
Let $(A,\mu,\alpha)$ be a right Hom-alternative algebra.  Then the identities
\begin{subequations}
\begin{align}
a'a_1' &= \alpha (a^2)',\label{1}\\
a^a &= 0 = a^b + b^a \label{3}
\end{align}
\end{subequations}
hold in $A$.  If, in addition, $A$ is multiplicative, then the identities
\begin{subequations}
\begin{align}
a' b_1' a_2' &= \alpha^2 \left[(ab)a_1\right]',\label{2}\\
a' b_1' c_2' + c' b_1' a_2' &= \alpha^2 \left[(ab)c_1 + (cb)a_1\right]' \label{2'}
\end{align}
\end{subequations}
hold in $A$.
\end{lemma}

\begin{proof}
The right Hom-alternative identity \eqref{rhomalt} expands to
\[
(xa)\alpha(a) = \alpha(x)(aa).
\]
The identity \eqref{1} is the operator form of the previous identity.  By right Hom-alternativity and \eqref{herstein} we have
\[
0 = -(x,a,a) = xa^a,
\]
which proves the first identity in \eqref{3}.  The other identity in \eqref{3} is obtained from the first identity by linearization, i.e., by replacing $a$ with $a+b$ and using the linearity of $a^b$ in both the $a$ and the $b$ variables (Lemma \ref{lem1}).  The identity \eqref{2} is the operator form of the right Hom-Moufang identity \eqref{rightmoufang}.  The identity \eqref{2'} is obtained from \eqref{2} by replacing $a$ with $a+c$.
\end{proof}

\emph{For the rest of this section, $(A,\mu,\alpha)$ denotes an arbitrary multiplicative right Hom-alternative algebra.  Also let $p$ denote the Hom-associator $(a,a,b)$.  By multiplicativity we have
\[
p_n = \alpha^n(p) = (a_n,a_n,b_n).
\]}

The following result provides a situation in which the composition of $a^b$ and $c_d$ is trivial.

\begin{lemma}
\label{lem3}
The identities
\begin{subequations}
\begin{align}
a^b({a_2}_{b_2}) &= 0,\label{5}\\
a^b({a_2}_{c_2}) + a^c({a_2}_{b_2}) &= 0\label{5'}
\end{align}
\end{subequations}
hold in $A$.
\end{lemma}

\begin{proof}
The identity \eqref{5'} is the linearization of the identity \eqref{5} with respect to $b$.  To prove \eqref{5}, we compute as follows:
\[
\begin{split}
a^b({a_2}_{b_2})
&= \left\{\alpha (ab)' - a'b_1'\right\} \left\{\alpha (a_2 b_2)' - b_2' a_3'\right\} \quad\text{(by \eqref{ab})}\\
&= \alpha (ab)' \alpha (ab)_2' + a'b_1'b_2'a_3' - \left\{\alpha (ab)' b_2' a_3' + a'b_1'\alpha (ab)_2'\right\}\\
&= \alpha^2 (ab)_1'(ab)_2' + a'\alpha(b_1^2)'a_3' - \left\{\alpha (ab)' b_2' a_3' + \alpha a_1'b_2'(ab)_2'\right\} \quad\text{(by \eqref{0} and \eqref{1})}\\
&= \alpha^3[(ab)_1^2]' + \alpha a_1'(b_1^2)'a_3'\\
&\relphantom{} - \alpha^3\left\{\left[(ab)b_1\right]a_2 + (a_1b_1)(ab)_1\right\}' \quad\text{(by \eqref{0}, \eqref{1}, and \eqref{2'})}\\
&= \alpha^3[(ab)_1^2]' + \alpha^3[(a_1b^2)a_2]' - \alpha^3\left\{(a_1b^2)a_2 + (ab)_1^2\right\}' \quad\text{(by \eqref{xyy} and \eqref{2})}\\
&= 0.
\end{split}
\]
This finishes the proof of the identity \eqref{5}.
\end{proof}

The following result expresses the operator $\alpha^2 p_k'$ in terms of the operator $c^d$.

\begin{lemma}
\label{lem8}
For each integer $k \geq 0$, the identity
\begin{equation}
\label{10}
\alpha^2p_k' = \alpha a_{k+1}^{(ba)_k} - a_k' a_{k+1}^{b_{k+1}}
\end{equation}
holds in $A$.
\end{lemma}

\begin{proof}
To prove \eqref{10}, it suffices to prove the case $k=0$:
\begin{equation}
\label{10b}
\alpha^2(a,a,b)' = \alpha a_1^{ba} - a'a_1^{b_1}
\end{equation}
Indeed, the desired identity \eqref{10} is obtained from \eqref{10b} by replacing $a$ and $b$ by $a_k$ and $b_k$, respectively.  To prove \eqref{10b}, we compute as follows:
\[
\begin{split}
\alpha^2(a,a,b)'
&= -\alpha^2 (a,b,a)' \quad\text{(by \eqref{righthomalt})}\\
&= \alpha^2 \left[a_1(ba)\right]' - \alpha^2\left[(ab)a_1\right]' \quad\text{(by \eqref{homassociator})}\\
&= \alpha^2\left[a_1(ba)\right]' - a'b_1'a_2' \quad\text{(by \eqref{2})}\\
&= \alpha^2\left[a_1(ba)\right]' - a'\left\{a_1^{b_1} + \alpha(b_1a_1)'\right\} \quad\text{(by \eqref{3})}\\
&= \alpha\left\{\alpha[a_1(ba)]' - a_1'(ba)_1'\right\} - a'a_1^{b_1} \quad\text{(by \eqref{0})}\\
&= \alpha a_1^{ba} - a'a_1^{b_1} \quad\text{(by \eqref{ab})}
\end{split}
\]
This finishes the proof of the identity \eqref{10b}.
\end{proof}

The following result expresses the operator $\alpha^2p_k'$ in terms of the operator $c_d$.

\begin{lemma}
\label{lem9}
For each integer $k \geq 0$, the identity
\begin{equation}
\label{10'}
\alpha^2p_k' = \alpha {a_{k+1}}_{(ba)_k} - {a_k}_{b_k} a_{k+2}'
\end{equation}
holds in $A$.
\end{lemma}

\begin{proof}
To prove \eqref{10'}, it suffices to prove the case $k=0$:
\begin{equation}
\label{10''}
\alpha^2(a,a,b)' = \alpha {a_1}_{ba} - a_b a_2'.
\end{equation}
Indeed, the desired identity \eqref{10'} is obtained from \eqref{10''} by replacing $a$ and $b$ by $a_k$ and $b_k$, respectively.  To prove \eqref{10''}, we compute as follows:
\[
\begin{split}
\alpha^2(a,a,b)'
&= \alpha^2\left[a_1(ba)\right]' - a'b_1'a_2' \quad
\text{(by the proof of Lemma \ref{lem8})}\\
&= \alpha^2\left[a_1(ba)\right]' - \left\{\alpha(ba)' + a_b\right\}a_2' \quad\text{(by \eqref{ab} and \eqref{3})}\\
&= \alpha\left\{\alpha[a_1(ba)]' - (ba)'a_2'\right\} - a_ba_2'\\
&= \alpha{a_1}_{ba} - a_ba_2' \quad\text{(by \eqref{ab})}.
\end{split}
\]
This finishes the proof of \eqref{10''}.
\end{proof}

The following result shows that the operator in Proposition \ref{prop:0} splits into two operators.

\begin{lemma}
\label{lem:d+e}
The identity
\begin{equation}
\label{d+e}
\begin{split}
a^bp'p_1'p_2'\alpha^6
&= -a^b \alpha \left({a_3}_{(ba)_2}\right) \alpha \left(a_6^{(ba)_5}\right) \left({a_8}_{b_8}\right) a_{10}'\\
&\relphantom{} - a^b \alpha \left({a_3}_{(ba)_2}\right) a_5' \left(a_6^{b_6}\right) \alpha \left({a_9}_{(ba)_8}\right)
\end{split}
\end{equation}
holds in $A$.
\end{lemma}

\begin{proof}
Using \eqref{0} repeatedly, we have:
\begin{equation}
\label{de1}
\begin{split}
a^bp'p_1'p_2'\alpha^6
&= a^b p' p_1' \alpha^6 p_8'\\
&= a^b p' \alpha^4 p_5' \alpha^2 p_8'\\
&= a^b \alpha^2 p_2' \alpha^2 p_5' \alpha^2 p_8'.
\end{split}
\end{equation}
We now compute the previous line in two steps.  First, we have:
\begin{equation}
\label{de2}
\begin{split}
a^b\alpha^2 p_2'
&= a^b \alpha^2 (a_2,a_2,b_2)'\\
&= a^b\left\{\alpha {a_3}_{(ba)_2} - {a_2}_{b_2}a_4'\right\} \quad\text{(by \eqref{10'})}\\
&= a^b\alpha {a_3}_{(ba)_2} \quad\text{(by \eqref{5})}.
\end{split}
\end{equation}
Next we have:
\begin{equation}
\label{de3}
\begin{split}
\left(\alpha^2 p_5'\right) \left(\alpha^2 p_8'\right)
&= \left\{\alpha a_6^{(ba)_5} - a_5' a_6^{b_6}\right\}
\left\{\alpha {a_9}_{(ba)_8} - {a_8}_{b_8}a_{10}'\right\} \quad\text{(by \eqref{10} and \eqref{10'})}\\
&= \alpha a_6^{(ba)_5}\alpha {a_9}_{(ba)_8}
+ a_5' a_6^{b_6}{a_8}_{b_8}a_{10}'
- \alpha a_6^{(ba)_5}{a_8}_{b_8}a_{10}'
- a_5' a_6^{b_6}\alpha {a_9}_{(ba)_8}\\
&= \alpha^2 a_7^{(ba)_6}{a_9}_{(ba)_8}
- \alpha a_6^{(ba)_5}{a_8}_{b_8}a_{10}'
- a_5' a_6^{b_6}\alpha {a_9}_{(ba)_8} \quad\text{(by \eqref{0} and \eqref{5})}\\
&= - \alpha a_6^{(ba)_5}{a_8}_{b_8}a_{10}'
- a_5' a_6^{b_6}\alpha {a_9}_{(ba)_8} \quad\text{(by \eqref{5})}.\\
\end{split}
\end{equation}
We now obtain the desired identity \eqref{d+e} by combining \eqref{de1}, \eqref{de2}, and \eqref{de3}.
\end{proof}

Our next objective is to show that each summand on the right-hand side of \eqref{d+e} is equal to $0$.  In order to do that, we need several more auxiliary identities.

In \eqref{5} above it was shown that the operator $a^b({a_2}_{b_2})$ is trivial.  The following result gives a computation of the operator $a_b(a_2^{b_2})$.

\begin{lemma}
\label{lem4}
The identity
\begin{equation}
\label{6}
a_b(a_2^{b_2}) = -\alpha^3\left([a,b], a_1, b_1\right)'
\end{equation}
holds in $A$, where $(,,)$ is the Hom-associator in $A$ \eqref{homassociator} and
\[
[a,b] = ab - ba.
\]
\end{lemma}

\begin{proof}
Starting from the definitions \eqref{ab}, we compute as follows:
\[
\begin{split}
a_b(a_2^{b_2})
&= \left\{\alpha(ab)' - b'a_1'\right\} \left\{\alpha (ab)_2' - a_2'b_3'\right\}\\
&= \alpha(ab)'\alpha(ab)_2' + b'a_1'a_2'b_3' - \left\{\alpha(ab)'a_2'b_3' + b'a_1'\alpha(ab)_2'\right\}\\
&= \alpha^2(ab)_1'(ab)_2' + b'\alpha(a_1^2)'b_3' - \alpha\left\{(ab)'a_2'b_3' + b_1'a_2'(ab)_2'\right\} \quad\text{(by \eqref{0} and \eqref{1})}\\
&= \alpha^3[(ab)_1^2]' + \alpha b_1'(a_1^2)'b_3'\\
&\relphantom{} - \alpha^3\left\{[(ab)a_1]b_2 + (b_1a_1)(ab)_1\right\}' \quad\text{(by \eqref{0}, \eqref{1}, and \eqref{2'})}\\
&= \alpha^3\left\{(ab)_1(a_1b_1)\right\}' + \alpha^3 \left\{[(ba)a_1]b_2\right\}'\\
&\relphantom{} - \alpha^3\left\{[(ab)a_1]b_2 + (ba)_1(a_1b_1)\right\}' \quad\text{(by \eqref{xyy} and \eqref{2})}\\
&= -\alpha^3\left\{([a,b]a_1)b_2 - [a,b]_1(a_1b_1)\right\}'\\
&= -\alpha^3\left([a,b],a_1,b_1\right)'.
\end{split}
\]
This finishes the proof of the identity \eqref{6}.
\end{proof}

The following result gives a variation of the identity \eqref{6}.

\begin{lemma}
\label{lem5}
The identity
\begin{equation}
\label{7}
a_b a_2' (a_3^{b_3}) = -\alpha^4\left([a,b]a_1, a_2, b_2\right)'
\end{equation}
holds in $A$.
\end{lemma}

\begin{proof}
Starting from the definitions \eqref{ab}, we have:
\[
\begin{split}
a_b a_2' (a_3^{b_3})
&= \alpha(ab)'a_2'\alpha(ab)_3' + b'a_1'a_2'a_3'b_4'
- \left\{\alpha(ab)'a_2'a_3'b_4' + b'a_1'a_2'\alpha(ab)_3'\right\}\\
&= \alpha^2(ab)_1'a_3'(ab)_3' + b'\alpha^2(a_1^2a_2)'b_4'\\
&\relphantom{} - \left\{\alpha(ab)'\alpha(a_2^2)'b_4' + b'\alpha(a_1^2)'\alpha(ab)_3'\right\} \quad\text{(by \eqref{0}, \eqref{1}, and \eqref{2})}\\
&= \alpha^4\left\{[(ab)_1a_2](ab)_2\right\}' + \alpha^2b_2'(a^3_1)'b_4'\\
&\relphantom{} - \alpha^2\left\{(ab)_1'(a_2^2)'b_4' + b_2'(a_2^2)'(ab)_3'\right\} \quad\text{(by \eqref{x4}, \eqref{0}, and \eqref{2})}\\
&= \alpha^4\left\{[(ab)a_1]_1(a_2b_2)\right\}' + \alpha^4\left\{(b_2a^3)b_3\right\}'\\
&\relphantom{} - \alpha^4\left\{[(ab)_1a^2_1]b_3 + (b_2a^2_1)(ab)_2\right\}' \quad\text{(by \eqref{2} and \eqref{2'})}\\
&= \alpha^4\left\{[(ab)a_1]_1(a_2b_2)\right\}' + \alpha^4\left\{([(ba)a_1]a_2)b_3\right\}'\\
&\relphantom{} - \alpha^4\left\{([(ab)a_1]a_2)b_3 + [(ba)a_1]_1(ab)_2\right\}' \quad\text{(by \eqref{xyy} and \eqref{rightmoufang})}\\
&= -\alpha^4\left\{[([a,b]a_1)a_2]b_3 - ([a,b]a_1)_1(a_2b_2)\right\}'\\
&= -\alpha^4\left([a,b]a_1, a_2, b_2\right)'.
\end{split}
\]
This proves the identity \eqref{7}.
\end{proof}

The following result gives situations in which the product of two Hom-associators is trivial.

\begin{lemma}
\label{lem6}
The identities
\begin{equation}
\label{8}
p_3 \left([a_2,b_2], a_3, b_3\right) = 0
\end{equation}
and
\begin{equation}
\label{9}
p_4 \left([a_2,b_2]a_3, a_4, b_4\right) = 0
\end{equation}
hold in $A$.
\end{lemma}

\begin{proof}
To prove \eqref{8} we compute as follows:
\[
\begin{split}
p_3 \left([a_2,b_2], a_3, b_3\right)
&= -aa^b \alpha^3 \left([a_2,b_2], a_3, b_3\right)' \quad\text{(by \eqref{4})}\\
&= aa^b({a_2}_{b_2})a_4^{b_4} \quad\text{(by \eqref{6})}\\
&= 0 \quad\text{(by \eqref{5})}.
\end{split}
\]
To prove \eqref{9} we compute as follows:
\[
\begin{split}
p_4 \left([a_2,b_2]a_3, a_4, b_4\right)
&= -aa^b\alpha^4 \left([a_2,b_2]a_3, a_4, b_4\right)' \quad\text{(by \eqref{4})}\\
&= aa^b\left({a_2}_{b_2}\right) a_4' \left(a_5^{b_5}\right) \quad\text{(by \eqref{7})}\\
&= 0 \quad\text{(by \eqref{5})}.
\end{split}
\]
This finishes the proof of the identity \eqref{9}.
\end{proof}

In the next two lemmas, it is shown that each summand on the right-hand side of \eqref{d+e} is $0$.

\begin{lemma}
\label{lem:d}
The identity
\begin{equation}
\label{d=0}
-a^b \alpha \left({a_3}_{(ba)_2}\right) \alpha \left(a_6^{(ba)_5}\right) \left({a_8}_{b_8}\right) a_{10}' = 0
\end{equation}
holds in $A$.
\end{lemma}

\begin{proof}
To prove \eqref{d=0}, we compute as follows:
\[
\begin{split}
& -a^b \alpha \left({a_3}_{(ba)_2}\right) \alpha \left(a_6^{(ba)_5}\right) \left({a_8}_{b_8}\right) a_{10}'\\
&= -\alpha \left(a_1^{b_1} {a_3}_{(ba)_2}\right) \alpha \left(a_6^{(ba)_5} {a_8}_{b_8}\right) a_{10}' \quad\text{(by \eqref{0})}\\
&= -\alpha\left(a_1^{ba} {a_3}_{b_3}\right) \alpha \left(a_6^{b_6} {a_8}_{(ba)_7}\right) a_{10}' \quad\text{(by \eqref{5'})}\\
&= -\left(\alpha a_1^{ba}\right) \left({a_3}_{b_3} a_5^{b_5}\right) \left(\alpha {a_8}_{(ba)_7}\right) a_{10}' \quad\text{(by \eqref{0})}\\
&= -\left(\alpha^2p' + a'a_1^{b_1}\right) \left({a_3}_{b_3} a_5^{b_5}\right) \left(\alpha^2p_7' + {a_7}_{b_7}a_9'\right) a_{10}' \quad\text{(by \eqref{10} and \eqref{10'})}\\
&= -\alpha^2p' \left({a_3}_{b_3} a_5^{b_5}\right) \alpha^2p_7' a_{10}' \quad\text{(by \eqref{5})}\\
&= \alpha^2p'\alpha^3 \left([a,b]_3, a_4, b_4\right)' \alpha^2 p_7' a_{10}' \quad\text{(by \eqref{6})}\\
&= \alpha^7p_5' \left([a,b]_4, a_5, b_5\right)_1' p_7' a_{10}' \quad\text{(by \eqref{0})}\\
&= \alpha^9 \left\{\left(p_5 ([a,b]_4,a_5,b_5)\right) p_6\right\}' a_{10}' \quad\text{(by \eqref{2})}\\
&= \alpha^9 \left\{(p_3([a_2,b_2],a_3,b_3))_2 p_6\right\}' a_{10}' \quad\text{(by \eqref{0})}\\
&= 0 \quad\text{(by \eqref{8})}.
\end{split}
\]
This finishes the proof of the identity \eqref{d=0}.
\end{proof}

\begin{lemma}
\label{lem:e}
The identity
\begin{equation}
\label{e=0}
- a^b \alpha \left({a_3}_{(ba)_2}\right) a_5' \left(a_6^{b_6}\right) \alpha \left({a_9}_{(ba)_8}\right) = 0
\end{equation}
holds in $A$.
\end{lemma}

\begin{proof}
To prove \eqref{e=0} we compute as follows:
\[
\begin{split}
& - a^b \alpha \left({a_3}_{(ba)_2}\right) a_5' \left(a_6^{b_6}\right) \alpha \left({a_9}_{(ba)_8}\right)\\
&= -\alpha \left(a_1^{b_1}{a_3}_{(ba)_2}\right) a_5' \left(a_6^{b_6}\right) \alpha \left({a_9}_{(ba)_8}\right) \quad \text{(by \eqref{0})}\\
&= \left(\alpha a_1^{ba}\right) \left({a_3}_{b_3} a_5' a_6^{b_6}\right) \alpha \left({a_9}_{(ba)_8}\right) \quad \text{(by \eqref{5'})}\\
&= \left(\alpha^2p' + a'a_1^{b_1}\right) \left({a_3}_{b_3} a_5' a_6^{b_6}\right) \left(\alpha^2p_8' + \left({a_8}_{b_8}\right) a_{10}'\right) \quad \text{(by \eqref{10} and \eqref{10'})}\\
&= \alpha^2p' \left({a_3}_{b_3} a_5' a_6^{b_6}\right) \alpha^2p_8' \quad\text{(by \eqref{5})}\\
&= -\alpha^2 p' \alpha^4 \left([a,b]_3a_4,a_5,b_5\right)' \alpha^2p_8' \quad\text{(by \eqref{7})}\\
&= -\alpha^8 p_6' \left([a,b]_4a_5,a_6,b_6\right)_1' p_8' \quad\text{(by \eqref{0})}\\
&= -\alpha^{10} \left\{\left(p_6 \left([a,b]_4a_5,a_6,b_6\right)\right) p_7\right\}' \quad\text{(by \eqref{2})}\\
&= -\alpha^{10} \left\{\left(p_4 \left([a_2,b_2]a_3, a_4, b_4\right)\right)_2 p_7\right\}' \quad\text{(by \eqref{0})}\\
&= 0 \quad \text{(by \eqref{9})}.
\end{split}
\]
This finishes the proof of the identity \eqref{e=0}.
\end{proof}

\begin{proposition}
\label{prop:0}
The identity
\[
a^bp'p_1'p_2'\alpha^6 = 0
\]
holds in $A$, where $p = (a,a,b)$.
\end{proposition}

\begin{proof}
Combine Lemmas \ref{lem:d+e}, \ref{lem:d}, and \ref{lem:e}.
\end{proof}

The proof of Theorem \ref{thm:mikheev} is now complete.


\end{document}